\documentclass[11pt, letterpaper]{amsart}
\usepackage{fullpage,amsfonts,amsmath,amssymb,amsthm,tikz}
\usetikzlibrary{matrix}
\usepackage{bm}
\usepackage[vcentermath]{youngtab}
\usepackage{young}
\usepackage[all,cmtip]{xy}
\usepackage{hyperref}
\newtheorem{thm}{Theorem}[section]
\newtheorem{cor}[thm]{Corollary}
\newtheorem{lem}[thm]{Lemma}
\newtheorem{prop}[thm]{Proposition}

\theoremstyle{definition}
\newtheorem{definition}{Definition}[section]
\newtheorem{example}{Example}[section]
\newtheorem*{remark}{Remark}

\newcommand{\CC}{\mathbb{C}}
\newcommand{\FF}{\mathbb{F}}
\newcommand{\KK}{\mathbb{K}}

\newcommand{\cB}{\mathcal{B}}

\newcommand{\Ind}{\mathrm{Ind}}
\newcommand{\GL}{\mathrm{GL}}

\newcommand{\Res}{\mathrm{Res}}

\newcommand{\UT}{\mathrm{UT}}

\newcommand{\One}{{1\hspace{-.14cm} 1}}

\def\adots{\mathinner{\mkern2mu\raise0pt\hbox{.}  
\mkern2mu\raise4pt\hbox{.}\mkern1mu
\raise7pt\vbox{\kern7pt\hbox{.}}\mkern1mu}}

\usetikzlibrary[positioning,patterns]

\tikzstyle{bsq}=[rectangle, draw, thick, minimum width=.75cm, minimum height=.75cm]

\begin{document}
\title{The irreducible unipotent modules of the finite general linear groups via tableaux}
\author{Scott Andrews}
\address{Department of Mathematics \\ Boise State University \\ Boise, ID 83725}
\email{scottandrews@boisestate.edu}

\begin{abstract} We construct the irreducible unipotent modules of the finite general linear groups using tableaux. Our construction is analogous to that of James (1976) for the symmetric groups, answering an open question as to whether such a construction exists. Our modules are defined over any field containing a nontrivial $p^\text{th}$ root of unity (where $p$ is the defining characteristic of the group). We show that our modules are isomorphic to those constructed by James (1984), although the two constructions utilize different approaches. Finally we look closer at the complex irreducible unipotent modules, providing motivation for our construction in the language of symmetric functions.
\end{abstract}
\maketitle

The standard construction of the irreducible representations of the symmetric group $S_n$ (initially due to James \cite{MR0417272}) uses the action of $S_n$ on Young tableaux to define, for each integer partition of $n$, a ``Specht module.'' In characteristic $0$, the Specht modules are (up to isomorphism) all of the irreducible $S_n$-modules; in other characteristics, the irreducible modules appear as quotients of the Specht modules. James shows in \cite{MR0439924} that the Specht module corresponding to the partition $\lambda$ has a basis indexed by the standard Young tableaux of shape $\lambda$.

\bigbreak

There is a principal in representation theory that information about the finite general linear group $GL_n(\FF_q)$ can be related to that of $S_n$ by ``setting $q=1$.'' In particular, there is a collection of irreducible representations of $\GL_n(\FF_q)$, known as ``unipotent representations,'' that one would expect to behave like the irreducible representations of $S_n$. In \cite{MR776229}, James constructs the unipotent modules of $\GL_n(\FF_q)$ over any field containing a nontrivial $p^\text{th}$ root of unity (where $p$ is the characteristic of $\FF_q$). The construction is quite different from the tableaux approach for the symmetric group, and in particular (as James notes in the introduction of \cite{MR776229}) the proofs do not translate to proofs for the symmetric group.

\bigbreak

In James' construction for the finite general linear groups, a collection of modules (which James also calls Specht modules) play a similar role to that of the Specht modules of the symmetric group. It is still an open problem to determine a basis for these modules in the finite general linear group case (see, for instance, \cite{MR2047443}). In providing a new construction of the irreducible unipotent modules, we hope to shed some light on this question.

\bigbreak

The key ingredient in James' construction for $S_n$ is the interaction between the two linear representations of $S_n$, the trivial representation $\One$ and the sign representation $\epsilon$. In particular, over a field of characteristic $0$, if $W_\lambda$ is a Young subgroup of shape $\lambda$ and $W_{\lambda'}$ is a Young subgroup of shape $\lambda'$, we have that
\[
        \langle \Ind_{W_\lambda}^{S_n}(\One), \Ind_{W_{\lambda'}}^{S_n}(\epsilon) \rangle = 1;
\]
the common irreducible component is the Specht module indexed by $\lambda$. The Specht module is constructed as an irreducible submodule of $\Ind_{W_\lambda}^{S_n}(\One)$ that contains a one-dimensional subspace on which $W_{\lambda'}$ acts as $\epsilon$.

\bigbreak

For the finite general linear group, the natural analogue to $\Ind_{W_\lambda}^{S_n}(\One)$ is $\Ind_{P_{\lambda'}}^{\GL_n(\FF_q)}(\One)$, where $P_{\lambda'}$ is the parabolic subgroup of $\GL_n(\FF_q)$ of shape $\lambda'$ (see Section ~\ref{preliminaries}). There are two reasonable analogues to $\Ind_{W_{\lambda'}}^{S_n}(\epsilon)$, however; these are the ``degenerate Gelfand--Graev characters'' of Zelevinsky \cite{MR643482} and the ``generalized Gelfand--Graev characters,'' which were initially constructed by Kawanaka \cite{MR803335} and recently studied by Thiem and the author \cite{andrewsthiem}. James' construction in \cite{MR776229} uses the degenerate Gelfand--Graev characters; we instead utilize the generalized Gelfand--Graev characters.

\bigbreak

Our approach is to label the boxes of Young diagrams with elements of $\FF_q^n$ rather than by integers. There is a natural action of $\GL_n(\FF_q)$ on these objects; we use this action to construct the irreducible unipotent modules. In Section~\ref{preliminaries}, we cover necessary background material on partitions and the finite general linear groups. Our construction is in Section~\ref{construction}, and in Section~\ref{complex} we look at the particular case where the ground field is $\CC$ and provide motivation for our construction.

\section{Preliminaries}\label{preliminaries}
\subsection{Partitions and tableaux}

Let $n$ be a positive integer; a \emph{partition} of $n$ is a sequence $\lambda = (\lambda_1,\lambda_2,\hdots,\lambda_k)$ of positive integers with $\lambda_1 \geq \lambda_2 \geq \hdots \geq \lambda_k$ and $\lambda_1+\lambda_2+\hdots +\lambda_k = n$. We write $\lambda \vdash n$ to indicate that $\lambda$ is a partition of $n$.

\bigbreak

There is a partial order on the set of partitions of $n$ with $\lambda \succeq \mu$ if and only if
\[
        \sum_{i=1}^k\lambda_i \geq \sum_{i=1}^k\mu_i
\]
for all $k$ (setting $\lambda_i=0$ if $\lambda$ has fewer than $i$ parts). This order is called the \emph{dominance order} on partitions.

\bigbreak

To each partition $\lambda$ we associate a \emph{Young diagram}, which is a left-justified array of blocks such that the number of blocks in the $i$th row is $\lambda_i$.

\begin{example} Let $\lambda = (4,3,1,1)$; then the Young diagram of shape $\lambda$ is
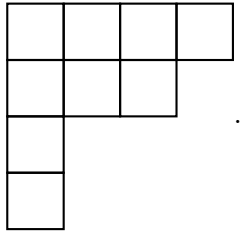
\[
		\begin{tikzpicture}[node distance=0 cm,outer sep = 0pt, baseline={([yshift=-.5ex]current bounding box.center)}]
	      \node[bsq] (11) at (   0,  0) {};
	      \node[bsq] (21) [below = of 11] {};
	      \node[bsq] (31) [below = of 21] {};
	      \node[bsq] (41) [below = of 31] {};
	      \node[bsq] (12) [right = of 11] {};
	      \node[bsq] (22) [below = of 12] {};
	      \node[bsq] (23) [right = of 22] {};
         	 \node[bsq] (13) [right = of 12] {};
          \node[bsq] (14) [right = of 13] {};
\end{tikzpicture}.
\]
\end{example}

The \emph{conjugate} of a partition $\lambda$, denoted $\lambda'$, is the partition defined by $\lambda_i' = |\{j \mid \lambda_j \geq i\}|$. Note that the Young diagram of $\lambda'$ is obtained from that of $\lambda$ by reflection about the diagonal.

\bigbreak

If $\lambda$ is a partition of $n$, a \emph{tableau} of shape $\lambda$ is a filling of the Young diagram of shape $\lambda$ by the integers from $1$ through $n$, each appearing exactly once. We say that a tableau is \emph{standard} if the entries increase along rows and columns.

\begin{example}
Let $\lambda = (4,3,1,1)$, and let 
\[
		T = \begin{tikzpicture}[node distance=0 cm,outer sep = 0pt, baseline={([yshift=-.5ex]current bounding box.center)}]
	      \node[bsq] (11) at (   0,  0) {$1$};
	      \node[bsq] (21) [below = of 11] {$3$};
	      \node[bsq] (31) [below = of 21] {$7$};
	      \node[bsq] (41) [below = of 31] {$9$};
	      \node[bsq] (12) [right = of 11] {$2$};
	      \node[bsq] (22) [below = of 12] {$4$};
	      \node[bsq] (23) [right = of 22] {$6$};
         	 \node[bsq] (13) [right = of 12] {$5$};
          \node[bsq] (14) [right = of 13] {$8$};
\end{tikzpicture} \quad \text{and} \quad
T' = \begin{tikzpicture}[node distance=0 cm,outer sep = 0pt, baseline={([yshift=-.5ex]current bounding box.center)}]
	      \node[bsq] (11) at (   0,  0) {$1$};
	      \node[bsq] (21) [below = of 11] {$4$};
	      \node[bsq] (31) [below = of 21] {$9$};
	      \node[bsq] (41) [below = of 31] {$6$};
	      \node[bsq] (12) [right = of 11] {$2$};
	      \node[bsq] (22) [below = of 12] {$8$};
	      \node[bsq] (23) [right = of 22] {$7$};
         	 \node[bsq] (13) [right = of 12] {$3$};
          \node[bsq] (14) [right = of 13] {$5$};
\end{tikzpicture};
\]
then $T$ and $T'$ are both tableaux of shape $\lambda$, but $T'$ is not standard as the pairs $(7,8)$ and $(6,9)$ violate the row-increasing and column-increasing conditions.
\end{example}

\subsection{The finite general linear groups}

Let $q$ be a power of a prime, and let $\FF_q$ be the finite field with $q$ elements. We are interested in $G = \GL_n(\FF_q)$, the group of invertible $n \times n$ matrices with entries in $\FF_q$.

\bigbreak

Let $\lambda$ be a partition of $n$, and let $T$ be the row-reading tableau of shape $\lambda$. We define
\begin{align*}
        P_\lambda &= \{g \in G \mid g_{ij} = 0 \text{ if } i \text{ is strictly below } j \text{ in } T\} \text{ and} \\ 
        U_\lambda & = \{g \in G \mid g_{ii} = 1 \text{ and } g_{ij} = 0 \text{ unless } i = j \text{ or }i \text{ is strictly above } j \text{ in } T\}.
\end{align*}

\begin{example} Let $\lambda = (4,3,1,1)$; then
\[
		T = \begin{tikzpicture}[node distance=0 cm,outer sep = 0pt, baseline={([yshift=-.5ex]current bounding box.center)}]
	      \node[bsq] (11) at (   0,  0) {$1$};
	      \node[bsq] (21) [below = of 11] {$5$};
	      \node[bsq] (31) [below = of 21] {$8$};
	      \node[bsq] (41) [below = of 31] {$9$};
	      \node[bsq] (12) [right = of 11] {$2$};
	      \node[bsq] (22) [below = of 12] {$6$};
	      \node[bsq] (23) [right = of 22] {$7$};
         	 \node[bsq] (13) [right = of 12] {$3$};
          \node[bsq] (14) [right = of 13] {$4$};
\end{tikzpicture},
\]
and we have
\[
		P_\lambda = \left\{\left(\begin{array}{ccccccccc}
		* & * & * & * & * & * & * & * & * \\
		* & * & * & * & * & * & * & * & * \\
		* & * & * & * & * & * & * & * & * \\
		* & * & * & * & * & * & * & * & * \\
		0 & 0 & 0 & 0 & * & * & * & * & * \\
		0 & 0 & 0 & 0 & * & * & * & * & * \\
		0 & 0 & 0 & 0 & * & * & * & * & * \\
		0 & 0 & 0 & 0 & 0 & 0 & 0 & * & * \\
		0 & 0 & 0 & 0 & 0 & 0 & 0 & 0 & * \end{array}\right)\in \GL_n(\FF_q)\right\}
\]
and
\[
		U_\lambda = \left\{\left(\begin{array}{ccccccccc}
		1 & 0 & 0 & 0 & * & * & * & * & * \\
		0 & 1 & 0 & 0 & * & * & * & * & * \\
		0 & 0 & 1 & 0 & * & * & * & * & * \\
		0 & 0 & 0 & 1 & * & * & * & * & * \\
		0 & 0 & 0 & 0 & 1 & 0 & 0 & * & * \\
		0 & 0 & 0 & 0 & 0 & 1 & 0 & * & * \\
		0 & 0 & 0 & 0 & 0 & 0 & 1 & * & * \\
		0 & 0 & 0 & 0 & 0 & 0 & 0 & 1 & * \\
		0 & 0 & 0 & 0 & 0 & 0 & 0 & 0 & 1 \end{array}\right)\in \GL_n(\FF_q)\right\}.
\]
We also define
\[
        P_\lambda^- = (P_\lambda)^t \quad \text{and} \quad U_\lambda^- = (U_\lambda)^t
\]
to be the transposes of $P_\lambda$ and $U_\lambda$.
\end{example}

\begin{remark} The groups $P_\lambda$ and $P_\lambda^-$ are parabolic subgroups of $G$, with unipotent radicals $U_\lambda$ and $U_\lambda^-$. In particular, the group of upper triangular matrices in $G$ given by 
\[
       B_n(\FF_q) = P_{(1^n)}
\]
is a Borel subgroup of $G$ with unipotent radical
\[
        \UT_n(\FF_q) = U_{(1^n)}.
\]
\end{remark}

We say that an irreducible module of $G$ (over some field) is \emph{unipotent} if it is a composition factor of $\Ind_{B_n(\FF_q)}^G(\One)$.

\section{The irreducible unipotent modules of $\GL_n(\FF_q)$}\label{construction}

The construction in this section is motivated by the construction of the irreducible representations of the symmetric group (see \cite{MR0417272,MR0439924,MR1824028}). Many of the results are similar to those found in \cite{MR818927}, which is not surprising as our construction produces isomorphic modules.

\bigbreak

Let $p$ be the characteristic of $\FF_q$, and let $\KK$ be a field that contains a nontrivial $p^\text{th}$ root of unity (in particular, the characteristic of $\KK$ cannot be $p$). For the remained of the paper, fix a nontrivial homomorphism $\theta:\FF_q^+ \to \KK^\times$.

\begin{definition} Let $\lambda$ be a partition of $n$, and let $T$ be a filling of the Young diagram of shape $\lambda$ with linearly independent elements of $\mathbb{F}_q^n$. We call $T$ an $\mathbb{F}_q^n$\emph{-tableau}.
\end{definition}

Note that $G$ acts on the set of $\mathbb{F}_q^n$-tableaux by left multiplication of the entries (considered as column vectors). If $T$ is an $\mathbb{F}_q^n$-tableau of shape $\lambda$, we obtain an ordered basis $\mathcal{B}(T)$ of $\mathbb{F}_q^n$ by numbering the entries of $T$ from top to bottom, then left to right.

\begin{example} Let $\lambda = (4,2^2,1)$; then the ordered basis $\mathcal{B}(T) = \{v_1,...,v_9\}$ corresponds to the tableau
\[
		T = \begin{tikzpicture}[node distance=0 cm,outer sep = 0pt, baseline={([yshift=-.5ex]current bounding box.center)}]
	      \node[bsq] (11) at (   0,  0) {$v_1$};
	      \node[bsq] (21) [below = of 11] {$v_2$};
	      \node[bsq] (31) [below = of 21] {$v_3$};
	      \node[bsq] (41) [below = of 31] {$v_4$};
	      \node[bsq] (12) [right = of 11] {$v_5$};
	      \node[bsq] (22) [below = of 12] {$v_6$};
	      \node[bsq] (32) [below = of 22] {$v_7$};
         	 \node[bsq] (13) [right = of 12] {$v_8$};
          \node[bsq] (14) [right = of 13] {$v_9$};
\end{tikzpicture} .
\]
\end{example}

To each $\FF_q^n$-tableaux $T$ we associate two subgroups of $G$, given by
\begin{align*}
		U(T) &= \{g \in G \mid g \cdot v_i -v_i \in \mathbb{F}_q\text{-span}\{v_j \mid v_j \text{ is strictly left of }v_i
		\text{ in }T\}\text{ for all }i\} \text{ and} \\
		P(T) &= \{g \in G \mid g \cdot v_i \in \mathbb{F}_q\text{-span}\{v_j \mid v_j \text{ is nonstrictly right of }v_i
		\text{ in }T\}\text{ for all }i\}.
\end{align*}

We remark that if $\mathcal{B}(T)$ is the standard ordered basis of $\mathbb{F}_q^n$, then $U(T) = U_{\lambda'}$ and $P(T) = P_{\lambda'}^-$.

\begin{example}\label{example1} Let $\lambda = (4,2^2,1)$, and let $\mathcal{B}(T)$ be the standard ordered basis. Then
\[
		T = \begin{tikzpicture}[node distance=0 cm,outer sep = 0pt, baseline={([yshift=-.5ex]current bounding box.center)}]
	      \node[bsq] (11) at (   0,  0) {$v_1$};
	      \node[bsq] (21) [below = of 11] {$v_2$};
	      \node[bsq] (31) [below = of 21] {$v_3$};
	      \node[bsq] (41) [below = of 31] {$v_4$};
	      \node[bsq] (12) [right = of 11] {$v_5$};
	      \node[bsq] (22) [below = of 12] {$v_6$};
	      \node[bsq] (32) [below = of 22] {$v_7$};
         	 \node[bsq] (13) [right = of 12] {$v_8$};
          \node[bsq] (14) [right = of 13] {$v_9$};
\end{tikzpicture} ,
\]
and we have
\[
		U(T) = \left\{\left(\begin{array}{ccccccccc}
		1 & 0 & 0 & 0 & * & * & * & * & * \\
		0 & 1 & 0 & 0 & * & * & * & * & * \\
		0 & 0 & 1 & 0 & * & * & * & * & * \\
		0 & 0 & 0 & 1 & * & * & * & * & * \\
		0 & 0 & 0 & 0 & 1 & 0 & 0 & * & * \\
		0 & 0 & 0 & 0 & 0 & 1 & 0 & * & * \\
		0 & 0 & 0 & 0 & 0 & 0 & 1 & * & * \\
		0 & 0 & 0 & 0 & 0 & 0 & 0 & 1 & * \\
		0 & 0 & 0 & 0 & 0 & 0 & 0 & 0 & 1 \end{array}\right)\right\}
\]
and
\[
		P(T) = \left\{\left(\begin{array}{ccccccccc}
		* & * & * & * & 0 & 0 & 0 & 0 & 0 \\
		* & * & * & * & 0 & 0 & 0 & 0 & 0 \\
		* & * & * & * & 0 & 0 & 0 & 0 & 0 \\
		* & * & * & * & 0 & 0 & 0 & 0 & 0 \\
		* & * & * & * & * & * & * & 0 & 0 \\
		* & * & * & * & * & * & * & 0 & 0 \\
		* & * & * & * & * & * & * & 0 & 0 \\
		* & * & * & * & * & * & * & * & 0 \\
		* & * & * & * & * & * & * & * & * \end{array}\right)\right\}.
\]

\end{example}

If $T$ is the $\FF_q^n$-tableau corresponding to the ordered basis $\cB(T)=\{v_1,v_2,\hdots,v_n\}$, let $X(T)$ be the set of pairs $(i,j)$ such that $v_i$ lies in the box directly to the left of $v_j$ in $T$. In Example~\ref{example1},
\[
        X(T) = \{(1,5),(2,6),(3,7),(5,8),(8,9)\}.
\] 
Define a linear character $\psi_T$ of $U(T)$ by
\[
		\psi_T(u) = \theta\left(\sum_{(i,j) \in X(T)} \text{the coefficient of } v_i \text{ in }uv_j \right),
\]
where $\theta$ is the fixed nontrivial homomorphism from $\FF_q^+$ to $\KK^\times$.

\begin{remark} The groups $P(T)$ and $U(T)$ are analogous to the row-stabilizer $R_T$ and the column-stabilizer $C_T$ of the symmetric group; the linear character $\psi_T$ is analogous to the sign character (see \cite{MR0417272,MR0439924,MR1824028}). 

\end{remark}

Consider the permutation $G$-module
\[
		\KK\text{-span}\{T \mid T \text{ is an } \FF_q^n\text{-tableau of shape }\lambda\};
\]
this module is isomorphic to the left regular module of $G$. We define
\begin{align*}
		m_T &= \sum_{p \in P(T)} pT \quad \text{and} \\
		e_T &= \sum_{u \in U(T)} \psi_T(u^{-1})m_{uT}.
\end{align*}

The following proposition is easy to verify directly.

\begin{lem}\label{properties} Let $T$ be any $\FF_q^n$-tableau.
\begin{enumerate}
\item For all $g \in G$, we have $U(gT) = gU(T)g^{-1}$ and $P(gT) = gP(T)g^{-1}$.
\item For all $g \in G$, we have $g \cdot m_T = m_{g \cdot T}$ and $g \cdot e_T = e_{gT}$.
\item For all $u \in U(T)$ and $g \in G$, we have $\psi_{gT}(gug^{-1}) = \psi_T(u)$.
\item For all $p \in P(T)$, we have $m_{pT} = m_T$.
\item For all $u \in U(T)$, we have $e_{uT} = \psi_T(u)e_T$.
\end{enumerate}
\end{lem}

Let
\[
		M^\lambda = \KK\text{-span}\{m_T \mid T \text{ is an }\FF_q^n\text{-tableau of shape }\lambda\}
\]
and
\[
		S^\lambda = \KK\text{-span}\{e_T \mid T \text{ is an }\FF_q^n\text{-tableau of shape }\lambda\};
\]
by part (2) of Lemma~\ref{properties}, $M^\lambda$ and $S^\lambda$ are both $G$-modules.

\begin{remark}
The module $M^\lambda$ is isomorphic to the permutation representation of $G$ on the set of $\lambda$-flags. This means that our module $M^\lambda$ is isomorphic to the module $M_{\lambda'}$ of James (as in \cite[10.1]{MR776229}).
\end{remark}

\begin{lem}
We have that
\[
        M^\lambda \cong \Ind_{P_{\lambda'}^-}^G(\One).
\]
\end{lem}

\begin{proof} By construction, we have that
\[
        M^\lambda \cong \Ind_{P(T)}^G(\One)
\]
for any $\FF_q^n$-tableaux $T$. Let $T$ correspond to the standard ordered basis of $\FF_q^n$; then $P(T) = P^-_{\lambda'}$.
\end{proof}

Let $W$ be the group of permutation matrices of $G$. The Bruhat decomposition allows us to write each element $g \in G$ in the form $g = uwb$, where $u \in UT_n(\FF_q)$, $w \in W$, and $b \in B^-_n(\FF_q)$. For a partition $\lambda$, let $W_\lambda = W \cap P_\lambda$ be the Young subgroup of shape $\lambda$.

\begin{lem}\label{groupintersection} Suppose that $T$ and $T'$ are $\FF_q^n$-tableaux of shape $\lambda$; then there exist $u \in U(T)$ and $p \in P(T')$ with $uT = pT'$ if and only if $U(T) \cap P(T') = \{1\}$.
\end{lem}

\begin{proof} By Lemma~\ref{properties}, we only need to consider the case where $T$ corresponds to the standard ordered basis. In this case, $U(T) = U_{\lambda'}$ and $P(T) = P_{\lambda'}^-$. Let $H = UT_n(\mathbb{F}_q)\cap P(T)$; then $H \cong UT_{\lambda_1'}(\mathbb{F}_q) \times \hdots \times UT_{\lambda'_k}(\mathbb{F}_q)$ and $UT_n(\FF_q) = U(T) \rtimes H$.

\bigbreak

Let $T' = gT$, and let $g = vwb$, with $v \in UT_n(\FF_q)$, $w \in W$, and $b \in B^-_n(\FF_q)$. Then
\[
		U(T) \cap P(T') = U(T) \cap gP(T)g^{-1} =U(T) \cap vwP(T)w^{-1}v^{-1}.
\]
As $UT_n(\FF_q)$ fixes $U(T)$ under conjugation, we have that $U(T) \cap vwP(T)w^{-1}v^{-1} = \{1\}$ if and only if $U(T) \cap w P(T)w^{-1} = \{1\}$; this occurs exactly when $w \in W_{\lambda'}$. We now have
\[
		g \in U(T)vwP(T) = U(T)vP(T) = U(T)P(T),
\]
as $W_{\lambda'}$ and $H$ are both contained in $P(T)$ and $v \in U(T) \rtimes H$. Write $g = u\tilde{p}$; then
\begin{align*}
		T' &= gT \\
		g\tilde{p}^{-1}g^{-1} T' &= g\tilde{p}^{-1}T \\
		g\tilde{p}^{-1}g^{-1} T' &= uT;
\end{align*}
set $p = g\tilde{p}^{-1}g^{-1}$. Conversely, suppose that $uT = pT'$, with $u \in U(T)$ and $p \in P(T')$; then $u = pg = g\tilde{p}$ for some $\tilde{p} \in P(T)$. It follows that $g \in U(T)P(T)$, hence $U(T) \cap P(T') = \{1\}$.
\end{proof}

\begin{lem}\label{nontrivialvalue} Suppose that $T$ and $T'$ are $\FF_q^n$-tableaux of shape $\lambda$ with $U(T)\cap P(T') \neq \{1\}$; then there exists $g \in U(T) \cap P(T')$ with $\psi_T(g) \neq 1$.
\end{lem}

\begin{proof} Once again, it suffices to consider the case where $T$ is the tableau corresponding to the standard ordered basis. As $U(T)\cap P(T') \neq \{1\}$, we have
\[
        U(T)\cap P(T') = U(T)\cap vwP(T)w^{-1}v^{-1} = v(U(T)\cap wP(T)w^{-1})v^{-1}
\]
for some $w \in W- W_{\lambda'}$ and $v \in UT_n(\FF_q)$. As $w \notin W_{\lambda'}$, there must be at least one pair $(i,j) \in X(T)$ (that is, with $v_{i}$ directly left of $v_{j}$ in $T$), but with $v_{i}$ appearing nonstrictly right of $v_{j}$ in $wT$. Then $1+\alpha e_{ij} \in U(T)\cap wP(T)w^{-1}$ for all $\alpha \in \mathbb{F}_q$; as $\theta$ is nontrivial, for some $\alpha \in \mathbb{F}_q$, we have $\psi_T(1+\alpha e_{ij}) \neq 1$.

\bigbreak

By the construction of $\psi_T$, we have that $\psi_T(1+\alpha e_{ij}) = \psi_T(v(1+\alpha e_{ij})v^{-1})$ for all $v \in UT_n(\FF_q)$. Let $g = v(1+\alpha e_{ij})v^{-1}$.
\end{proof}

For an $\FF_q^n$-tableau $T$, define an element $k_T \in \KK G$ by
\[
		k_T = \sum_{u \in U(T)} \psi_T(u^{-1})u.
\]
The following two lemmas describe how $k_T$ acts on $m_{T'}$ for certain $\FF_q^n$-tableaux $T'$.
\begin{lem}\label{onedimension} Let $T$ and $T'$ be $\FF_q^n$-tableaux of the same shape; then $k_Tm_{T'} \in \KK e_T$.
\end{lem}

\begin{proof} First suppose that $U(T) \cap P(T') = \{1\}$; by Lemma~\ref{groupintersection}, there exists $u \in U(T)$ and $p \in P(T')$ with $uT = pT'$. By Lemma~\ref{properties}, we have that
\[
        k_Tm_{T'} = k_T m_{pT'} = k_T m_{uT} = \psi_T(u)k_Tm_T = \psi_T(u)e_T.
\]
Suppose $U(T) \cap P(T') \neq \{1\}$; by Lemma~\ref{nontrivialvalue}, there exists $g \in U(T) \cap P(T')$ with $\psi_T(g) \neq 1$. We have that
\[
        k_Tm_{T'} = k_T m_{gT'} = \psi_T(g)k_Tm_{T'},
\]
hence $k_Tm_{T'} = 0$.
\end{proof}

\begin{lem}\label{differentshapes}
Let $\lambda$ and $\mu$ be partitions of $n$. If $T$ is an $\FF_q^n$-tableaux of shape $\lambda$ and $T'$ is an $\FF_q^n$-tableaux of shape $\mu$, then $k_Tm_{T'}=0$ unless $\mu \succeq \lambda$.
\end{lem}

\begin{proof}
Suppose that $\mu \not\succeq \lambda$. It suffice to show that this is true when $T$ is the $\FF_q^n$-tableaux of shape $\lambda$ corresponding to the standard ordered basis of $\FF_q^n$. Let $g \in G$ be such that $g^{-1}T'$ is the $\FF_q^n$-tableaux of shape $\mu$ corresponding to the standard ordered basis of $\FF_q^n$. We have that
\[
        U(T) \cap P(T') = U(T) \cap g P(g^{-1}T')g^{-1}  = v(U(T) \cap w P(g^{-1}T')w^{-1})v^{-1},
\]
where $v \in UT_n(\FF_q)$, $w \in W$, and $vwP(g^{-1}T') = gP(g^{-1}T')$.

\bigbreak

As $\mu \not\succeq\lambda$, there must be at least one pair $(i,j) \in X(T)$ (that is, with $v_{i}$ directly left of $v_{j}$ in $T$), but with $v_i$ nonstrictly right of $v_j$ in $wg^{-1}T'$. Then $1+\alpha e_{ij} \in U(T)\cap wP(g^{-1}T')w^{-1}$ for all $\alpha \in \mathbb{F}_q$; as $\theta$ is nontrivial, for some $\alpha \in \mathbb{F}_q$, we have $\psi_T(1+\alpha e_{ij}) \neq 1$.

\bigbreak

By the construction of $\psi_T$, we have that $\psi_T(1+\alpha e_{ij}) = \psi_T(v(1+\alpha e_{ij})v^{-1})$ for all $v \in UT_n(\FF_q)$. In other words, there is an element of $x \in U(T) \cap P(T')$ with $\psi_T(x) \neq 1$. Then
\[
        k_Tm_{T'} = k_T m_{xT'} = \psi_T(x)k_Tm_{T'},
\]
hence $k_Tm_{T'} = 0$.
\end{proof}

One consequence of Lemma~\ref{onedimension} is the following proposition.

\begin{prop}\label{indecomposable}
The module $S^\lambda$ is indecomposable.
\end{prop}

\begin{proof} Suppose that $S^\lambda = A \oplus B$ is a decomposition of $S^\lambda$ into a direct sum of $G$-modules. Let $T$ be any $\FF_q^n$-tableaux of shape $\lambda$; by Lemma~\ref{onedimension},
\[
        k_TS^\lambda \subseteq \KK e_T.
\]
At the same time,
\begin{align*}
        k_Te_T &= \sum_{u \in U(T)} \psi_T(u^{-1})ue_T \\
        & = \sum_{u \in U(T)} \psi_T(u^{-1})\psi_T(u)e_T \\ 
        &= |U(T)|e_T.
\end{align*}
As the characteristic of $\KK$ does not divide $|U(T)|$, we have that in fact
\[
        \KK e_T = k_TS^\lambda = k_T A \oplus k_T B,
\]
and $e_T \in k_T A$ or $e_T \in k_T B$. We may assume that $e_T \in k_T A$; as $A$ is a $\KK G$-module, $e_T \in A$. But $e_T$ generates $S^\lambda$ as a $\KK G$-module, hence $A = S^\lambda$.
\end{proof}

There is an immediate corollary of Proposition~\ref{indecomposable}.

\begin{cor} If $\textup{char}(\KK)$ does not divide $|G|$, then $S^\lambda$ is irreducible.  In particular, if $\KK$ has characteristic 0, then $S^\lambda$ is irreducible.
\end{cor}

Given an $\FF_q^n$-tableau $T$, let $\overline{T}$ be the $\FF_q^n$-tableau obtained by replacing the entry $v_i$ with $-v_i$ if $v_i$ is in an odd column of $T$ and fixing the entry $v_i$ if $v_i$ is in an even column of $T$. For example, if
\[
        T = \begin{tikzpicture}[node distance=0 cm,outer sep = 0pt, baseline={([yshift=-.5ex]current bounding box.center)}]
	      \node[bsq] (11) at (   0,  0) {$v_1$};
	      \node[bsq] (21) [below = of 11] {$v_2$};
	      \node[bsq] (31) [below = of 21] {$v_3$};
	      \node[bsq] (41) [below = of 31] {$v_4$};
	      \node[bsq] (12) [right = of 11] {$v_5$};
	      \node[bsq] (22) [below = of 12] {$v_6$};
	      \node[bsq] (32) [below = of 22] {$v_7$};
         	 \node[bsq] (13) [right = of 12] {$v_8$};
          \node[bsq] (14) [right = of 13] {$v_9$};
\end{tikzpicture} , \quad \text{then} \quad
\overline{T} = \begin{tikzpicture}[node distance=0 cm,outer sep = 0pt, baseline={([yshift=-.5ex]current bounding box.center)}]
	      \node[bsq] (11) at (   0,  0) {$-v_1$};
	      \node[bsq] (21) [below = of 11] {$-v_2$};
	      \node[bsq] (31) [below = of 21] {$-v_3$};
	      \node[bsq] (41) [below = of 31] {$-v_4$};
	      \node[bsq] (12) [right = of 11] {$v_5$};
	      \node[bsq] (22) [below = of 12] {$v_6$};
	      \node[bsq] (32) [below = of 22] {$v_7$};
         	 \node[bsq] (13) [right = of 12] {$-v_8$};
          \node[bsq] (14) [right = of 13] {$v_9$};
\end{tikzpicture}.
\]

\begin{lem}\label{tbar} For all $\FF_q^n$-tableaux $T$, we have that $U(\overline{T}) = U(T)$, $\overline{T}\in P(T)\cdot T$, and $\psi_{\overline{T}}(u) = \psi_T(u^{-1})$.
\end{lem}

\begin{proof} The first two claims are trivial. For the third, recall that
\[
		\psi_{T}(u) = \theta\left(\sum_{(i,j) \in X(T)} \text{the coefficient of } v_i \text{ in }uv_j \right),
\]
where once again $X(T)$ is the set of pairs $(i,j)$ such that $v_i$ lies in the box directly to the left of $v_j$ in $T$. If $(i,j) \in X(T)$, then exactly one of $v_i$ or $v_j$ has its sign switched in $\overline{T}$. It follows that
\[
        \psi_{\overline{T}}(u) = \theta\left(-\sum_{(i,j) \in X(T)} \text{the coefficient of } v_i \text{ in }uv_j \right) = \psi_{T}(u^{-1}).
\]
\end{proof}

We define a bilinear form on $M^\lambda$ by
\[
        [m_T,m_{T'}] = \delta_{m_T,m_{T'}},
\]
and extending by linearity. We remark that $\delta_{m_T,m_{T'}}$ is not the same as $\delta_{T,T'}$, as $m_T = m_{T'}$ exactly when $T' \in P(T) \cdot T$.

\begin{prop}\label{submodule} Let $V$ be a submodule of $M^\lambda$; then either $S^\lambda \subseteq V$ or $V \subseteq (S^\lambda)^\perp$.
\end{prop}

\begin{proof} First suppose that there exists $x \in V$ and an $\FF_q^n$-tableau $T$ such that $k_T x  \neq 0$; then by Lemma~\ref{onedimension}, we have $e_T \in V$, hence $S^\lambda \subseteq V$.

\bigbreak

Conversely suppose that, for all $x \in V$ and all $\FF_q^n$-tableaux $T$, we have $k_T x = 0$. Then
\begin{align*}
        [x,e_T] 
        &= \sum_{u \in U(T)}\psi_T(u^{-1})[x,um_T] \\
        & = \sum_{u \in U(T)}\psi_T(u^{-1})[u^{-1}x,m_T] \\
        & = [k_{\overline{T}}x, m_T] \\
        & = [0,m_T] = 0,
\end{align*}
as the bilinear form is $G$-invariant. It follows that $V \subseteq (S^\lambda)^\perp$.
\end{proof}

When constructing the irreducible representations of the symmetric groups, it is possible to have $S^\lambda \subseteq (S^\lambda)^\perp$; that will not be the case, however, with the finite general linear groups.

\begin{lem}\label{basisnotperp} If $T$ is any $\FF_q^n$-tableaux, then $e_T \notin (S^\lambda)^\perp$. In particular, $S^\lambda \not\subseteq (S^\lambda)^\perp$.
\end{lem}

\begin{proof}
We have that
\begin{align*}
        [e_T,e_{\overline{T}}]
        & = \sum_{u_1,u_2 \in U(T)} \psi_T(u_1^{-1})\psi_{\overline{T}}(u_2^{-1})[m_{u_1T},m_{u_2\overline{T}}] \\
        & = \sum_{u \in U(T)} \psi_T(u^{-1})\psi_{\overline{T}}(u^{-1}) \\
        & = |U(T)|,
\end{align*}
which is not $0$ as $\text{char}(\KK)$ does not divide $q$.
\end{proof}

The following corollary is an immediate consequence of Proposition~\ref{submodule}.

\begin{cor} We have the following.
\begin{enumerate}
\item $S^\lambda \cap (S^\lambda)^\perp$ is the unique maximal submodule of $S^\lambda$.
\item The $G$-module $S^\lambda/(S^\lambda \cap (S^\lambda)^\perp)$ is irreducible.
\end{enumerate}
\end{cor}

Define $D^\lambda = S^\lambda/(S^\lambda \cap (S^\lambda)^\perp)$.

\begin{prop}\label{dlambdaproperties}
Let $\lambda$ and $\mu$ be partitions of $n$; then we have the following.
\begin{enumerate}
\item If $D^\lambda$ is a composition factor of $M^\mu$, then $\lambda \preceq \mu$.
\item $D^\lambda$ is a composition factor of $M^\lambda$.
\item If $D^\lambda \cong D^\mu$, then $\lambda = \mu$.
\end{enumerate}
\end{prop}

\begin{proof}
Suppose that $D^\lambda$ is a composition factor of $M^\mu$; then we have a nonzero $G$-module homomorphism
\[
       \varphi : D^\lambda \to M^\mu/V
\]
for some submodule $V$ of $M^\mu$. As the elements $e_T+S^\lambda \cap (S^\lambda)^\perp$ generate $D^\lambda$, for some $\FF_q^n$-tableau $T$ we have $\varphi(e_T+S^\lambda \cap (S^\lambda)^\perp) \neq 0$. Note that $k_Te_T$ is a nonzero multiple of $e_T$, hence $k_T\varphi(e_T+S^\lambda \cap (S^\lambda)^\perp) \neq 0$. By Lemma~\ref{differentshapes}, $\lambda \preceq\mu$.

\bigbreak

Claim (2) follows directly from the definition of $D^\lambda$, and (3) is an immediate consequence of (1) and (2).
\end{proof}

In \cite{MR776229}, James constructs a collection of irreducible modules of $G$, one for each partition of $n$. James denotes these modules by $D_\lambda$, and shows the following.
\begin{enumerate}
\item The modules $D_\lambda$ are the unipotent modules of $G$. In other words, these modules are exactly the composition factors of $\Ind_B^G(\One)$, up to isomorphism.
\item If $D_\lambda \cong D_\mu$, then $\lambda = \mu$.
\item The module $D_\lambda$ is a composition factor of $\Ind_{P_\lambda}^G(\One)$.
\item Every composition factor of $\Ind_{P_\mu}^G(\One)$ is isomorphic to $D_\lambda$ for some $\lambda \succeq \mu$.
\end{enumerate}

Note that the modules $D_\lambda$ are uniquely characterized (up to isomorphism) by properties (2)--(4). As $M^\lambda \cong \Ind_{P_{\lambda'}^-}^G(\One)$ and $\Ind_{P_{\lambda'}^-}^G(\One) \cong \Ind_{P_{\lambda'}}^G(\One)$ (see \cite[14.7]{MR776229}), by Proposition~\ref{dlambdaproperties} we have the following.

\begin{cor} We have that $D^\lambda \cong D_{\lambda'}$. In particular, the $D^\lambda$ are the irreducible unipotent modules of $G$.
\end{cor}
\begin{remark}
The indexing of our modules and those of James differs by the transpose of the partition. Our indexing is chosen to match the convention of Green \cite{MR0072878} and MacDonald \cite{MR1354144}.
\end{remark}

\section{The irreducible unipotent modules over the complex numbers}\label{complex}

In this section we consider the specific case $\KK = \CC$, in which the modules $S^\lambda$ are in fact irreducible. First, we recall some facts about the character theory of $S_n$ and $\GL_n(\FF_q)$, and the connection to symmetric functions.

\bigbreak

The irreducible complex characters of $S_n$ are indexed by the partitions of $n$; let $\psi^\lambda$ denote the irreducible character of $S_n$ corresponding to $\lambda$, as in \cite{MR1354144}. Note that $\psi^{(n)}$ is the trivial character and $\psi^{(1^n)}$ is the sign character (denoted $\epsilon$). The following proposition follows from the Pieri rules (see \cite[I.5]{MR1354144}).

\begin{prop}
Let $\lambda$ be a partition of $n$, and let $W_\lambda$ be a Young subgroup of $S_n$ of shape $\lambda$. Then
\[
        \Ind_{W_\lambda}^{S_n}(\One) = \sum_{\mu \succeq \lambda} K_{\mu\lambda}\psi^\mu \quad \text{and}\quad \Ind_{W_\lambda}^{S_n}(\epsilon) = \sum_{\mu' \succeq \lambda} K_{\mu'\lambda}\psi^\mu,
\]
where the $K_{\mu\lambda}$ are the Kostka numbers (see \cite[I.6.4]{MR1354144}).
\end{prop}

In particular, if $W_\lambda$ is a Young subgroup of shape $\lambda$ and $W_{\lambda'}$ is a Young subgroup of shape $\lambda'$, we have
\[
        \langle \Ind_{W_\lambda}^{S_n}(\One), \Ind_{W_{\lambda'}}^{S_n}(\epsilon) \rangle = 1,
\]
and the common irreducible constituent is $\psi^\lambda$. A standard technique, in the case of $S_n$, is to construct an irreducible submodule $V^\lambda$ of $\Ind_{W_\lambda}^{S_n}(\One)$, and show that there is a one-dimensional subspace of $V^\lambda$ on which $W_{\lambda'}$ acts as $\epsilon$. In other words,
\[
       \langle \Ind_{W_\lambda}^{S_n}(\One), V^\lambda \rangle \neq 0
\]
and
\[
        \langle \Ind_{W_{\lambda'}}^{S_n}(\epsilon), V^\lambda \rangle
        = \langle \epsilon, \Res_{W_{\lambda'}}^{S_n}(V^\lambda) \rangle \neq 0.
\]
This forces $V^\lambda$ to be a module affording the character $\psi^\lambda$.

\bigbreak

In order to apply a similar approach to construct the irreducible unipotent modules of $\GL_n(\FF_q)$, we first need to determine modules that are analogous to $\Ind_{W_\lambda}^{S_n}(\One)$ and $\Ind_{W_{\lambda'}}^{S_n}(\epsilon)$. The irreducible unipotent characters of $\GL_n(\FF_q)$ are indexed by partitions of $n$; let $\chi^\lambda$ be the irreducible unipotent character corresponding to $\lambda$, as in \cite{MR0072878,MR1354144}. Note that $\chi^{(1^n)}$ is the trivial character and $\chi^{(n)}$ is the Steinberg character. As the trivial character is indexed by the transpose of $(n)$ (rather than by $(n)$ as with the symmetric group), we should expect our results to be transposed to some extent. The next proposition follows from the Pieri rules.

\begin{prop} Let $P_\lambda$ be the parabolic subgroup of $G$ of shape $\lambda$; then we have
\[
        \Ind_{P_\lambda}^G(\One) = \sum_{\mu' \succeq \lambda} K_{\mu'\lambda}\chi^\mu.
\]
\end{prop}

Both James' construction and our construction use $\Ind_{P_{\lambda'}}^G(\One)$ as the analogue of $\Ind_{W_\lambda}^{S_n}(\One)$.

\bigbreak

Let $\Psi^\lambda$ be the \emph{degenerate Gelfand--Graev character} corresponding to the partition $\lambda$, as in \cite{MR643482}. The character $\Psi^\lambda$ is obtained by inducing a linear character of $UT_n(\FF_q)$ that is trivial on certain root subgroups determined by the partition $\lambda$.

\begin{prop}[{\cite{MR643482}}] If $\lambda$ and $\mu$ are partitions of $n$, we have that
\[
       \langle \Psi^\lambda, \chi^\mu \rangle = K_{\mu \lambda}. 
\]
\end{prop}

It follows that $\langle \Ind_{P_{\lambda'}}^G(\One),\Psi^\lambda \rangle = 1$; the construction of James \cite{MR776229} uses $\Psi^\lambda$ as the analogue of $\Ind_{W_{\lambda'}}^{S_n}(\epsilon)$.

\bigbreak

In \cite{MR803335}, Kawanaka constructs the \emph{generalized Gelfand--Graev characters} of a reductive group over a finite field, with each character associated to a nilpotent orbit of the corresponding Lie algebra. In the case of $\GL_n(\FF_q)$, the nilpotent orbits are indexed by the partitions of $n$; let $\Gamma^\lambda$ be the generalized Gelfand--Graev character indexed by the partition $\lambda$. In \cite{andrewsthiem}, Thiem and the author show that $\Gamma^\lambda$ can be obtained by inducing a linear character from $U_{\lambda'}$, and calculate the multiplicities of the unipotent characters in $\Gamma^\lambda$.

\begin{thm}[{\cite{andrewsthiem}}]
Let $\lambda$ and $\mu$ be partitions of $n$; then
\[
        \langle \Gamma^\lambda, \chi^\mu \rangle = K_{\mu\lambda}(q),
\]
where $K_{\mu\lambda}(q)$ is the Kostka polynomial (see \cite[III.6]{MR1354144}). 
\end{thm}

It follows that $\langle \Ind_{P_{\lambda'}}^G(\One),\Gamma^\lambda \rangle = 1$. Furthermore, note that $K_{\mu\lambda}(1) = K_{\mu\lambda}$; in other words, by setting $q=1$ we see that $\Gamma^\lambda$ is another analogue of $\Ind_{W_{\lambda'}}^{S_n}(\epsilon)$. The following proposition follows directly from our construction and the construction of $\Gamma^\lambda$ in \cite{andrewsthiem}, and connects our modules $S^\lambda$ to the generalized Gelfand--Graev characters.

\begin{prop} Let $T$ be any $\FF_q^n$-tableau of shape $\lambda$; then
\[
		\text{Ind}_{U(T)}^G(\psi_T) = \Gamma^\lambda.
\]
\end{prop}

We can now show that our modules $S^\lambda$ afford the characters $\chi^\lambda$.

\begin{thm} Let $\chi_{S^\lambda}$ be the character afforded by $S^\lambda$; then $\chi_{S^\lambda} =  \chi^{\lambda}$.
\end{thm}

\begin{proof} As $S^\lambda$ is a submodule of $M^\lambda$, we have that
\[
		\langle \chi_{S^\lambda}, \text{Ind}_{P(T)}^G(\One) \rangle = \langle \chi_{S^\lambda}, \text{Ind}_{P_{\lambda'}}^G(\One) \rangle \neq 0.
\]
At the same time, $\mathbb{C}e_T$ is a $U(T)$-module that affords the character $\psi_T$. This means that
\[
		\langle \chi_{S^\lambda}, \Gamma^\lambda \rangle = \langle \chi_{S^\lambda}, \text{Ind}_{U(T)}^G(\psi_T) \rangle = \langle \text{Res}_{U(T)}^G(\chi_{S^\lambda}), \psi_T \rangle \neq 0.
\]
As $\chi^{\lambda}$ is the only common irreducible constituent of $\text{Ind}_{P_{\lambda'}}^G(\One)$ and $\Gamma^\lambda$, we must have $\chi_{S^\lambda} =  \chi^{\lambda}$.
\end{proof}

Finally, we note that we can use our results to identify an element of $\CC G$ that generates a module affording $\chi^\lambda$.

\begin{cor} Let $\lambda$ be a partition of $n$, and let $\varphi_\lambda$ be the linear character of $U_{\lambda'}$ that induces to $\Gamma^\lambda$. Then if
\[
        e = \sum_{\substack{u \in U_{\lambda'} \\ p \in P_{\lambda'}^-}} \varphi_\lambda(u^{-1})up,
\]
we have that $\mathbb{C}Ge$ is an irreducible $G$-module that affords the character $\chi^{\lambda}$.
\end{cor}

\bibliography{bibfile}
	\bibliographystyle{plain}

\end{document}